\documentclass[req no]{amsart}[16pt]
\usepackage{amsmath,amssymb}
\usepackage{graphicx}
\large\normalsize

\theoremstyle{plain}
\newtheorem{theorem}{Theorem}[section]

\newtheorem{cor}[theorem]{Corollary}

\theoremstyle{definition}
\newtheorem{definition}[theorem]{Definition}
\newtheorem{example}[theorem]{Example}
\newtheorem{remark}[theorem]{Remark}

\numberwithin{equation}{section}

\def\m{\mbox}
\bf
\def\D{\displaystyle}

\def\ma{\m{ and }}

\def\mo{\m{ or }}

\def\mfa{\m{ for all }}
\def\mf{\m{ for }}
\def\ii{\m{ if and only if }}

\def\o{\omega}
\def\O{\Omega}
\def\D{\mathcal{D}}

\def\h{\mathcal{H}}
\def\T {\mathcal{T}}
\def\E {\mathcal{E}}
\def\L {\mathcal{L}}
\def\M {\mathcal{M}}
\def\F {\mathcal{F}}
\def\s{\mathcal{S}}
\def\A{\mathcal{A}}

\def\W{\mathcal{W}}
\def\C{\mathcal{C}}
\def\U{\mathcal{U}}

\def\lni{\lim\limits_{n\to\infty}}

\begin{document}

\thispagestyle{empty} \setcounter{page}{1}

\begin{title}[ II Limit $W^*$ Algebra $M^{\Omega}$ ]{\mbox{}\\
[1cm]Contractions, Cogenerators, and Weak Stability}
\author{Robert O'Brien}
\end{title}

\vspace{.2cm}

\maketitle \thispagestyle{empty}

\vspace{-.2cm}
\begin{center}
Department of Mathematics,\\ The Catholic University of America(Retired),\\
Washington, DC 20064, USA.\\
\end{center}

\noindent{\bf Keywords:} Hilbert Space, Unitary Group, Weak
Stability, Almost Weak Stability, Weakly - Wandering, Contraction, Semigroup. \noindent{\bf AMS (MOS) Subject
Classification.} 47A20, 47A35

\begin{abstract} A $C_{0}$ contraction semigroup $\mathcal{T}$ on a Hilbert space $\h$ and its cogenerator $T$ define a $W^{*} - algebra, ~\M^{\O}$ - the \emph{limit algebra} - which determines the structure of the subspace of \emph{weakly Poisson recurrent (wPr)} vectors and gives a necessary and sufficient condition for $\mathcal{T}$ and $T$ to be \emph{weakly stable equivalent}.
\end{abstract}

\section{Introduction and Summary}
\subsection{Introduction}
Let $\mathcal{T} =\{T_{t}:t \geq 0\}$ be a $C_{0}$ contraction semigroup on a Hilbert space $\h$, $T$ its cogenerator, $\mathcal{D} = \{T^{n}:0\leq n<+\infty\}$. $\T$ and $\D$ have a common space of \emph{flight}( or \emph{almost weakly stable}[2,2.22] ) vectors $\h_{0}$. Assume $\h = \h_{0}$.

A commutative contraction semigroup $\s$ splits $\h_{0}$.  $\h_{0}= \h_{m}(\s) \oplus \h_w(\s)$ [9, Theorem 2.5] where $\h_{m}(\s)$ is the space of \emph{weakly Poisson recurrent (wPr)} vectors and $\h_{w}(\s)$ is the \emph{weakly - stable} subspace. $\T$ and $\D$ define dynamical systems on $\h$.\footnote{Dynamical system - the action of a commutative contraction semigroup $\s$ on a Hilbert space $\h$.} We shall examine the interaction of the limit operators of $\T$ and $\D$  and the limit states of these systems. These limit operators generate a $W^{*} - algebra, ~\M^{\O}$ - the \emph{limit algebra} - which determines the structure of the subspace $\h_{m}$ and the interaction of the limit states. $\M^{\O}$ also determines a necessary and sufficient condition[Section 5] for equivalence of weak stability of $\mathcal{T}$ and $\D$ i.e., $\h_w(\T)= \h_w(\D)$ . Reference [2] motivated this research with the open question [2, 2.23].

\subsection{Preliminaries} We use the definitions and notation of [9, 10]. $\T \ma \mathcal{D}$ have a common unitary subspace $\h_{u}\subset \h_{0} $. $\h_{u}$ is closed and reducing and $ \h_{m} \subset \h_{u}$ for both $\T \ma \mathcal{D}$. $\U$ is the unitary group defined on $\h_{u}$ by $\T$ and $\C$  is the corresponding group on $\h_{u}$ for $\mathcal{D}$ [9, 10]. $\T \ma \mathcal{D}$ each split $\h_{0}$,  $\h_{m}(\T)=\h_{m}(\U)$ and $\h_{m}(\D)=\h_{m}(\C)$. $\T = \U P_{m_{\T}}\oplus \T P_{w_{\T}}$ and similarly for $\D$.

\subsection{Example: A Weakly Stable Equivalent $C_{0}$ Contraction Semigroup and Cogenerator} [12, 1.3].
Take the Hilbert space $\h$ to be $\h = L^{2}([0, 1] \oplus  l^{2}{\mathbb{Z}}^{+})$. [9, 10] define for a strictly increasing, continuous-singular function $F$ on $[0, 1]$ a spectral family $\{F_\theta:0\leq\theta< 1\}$ with unitary operator $U = \int\limits^{1}_{0} e^{2 \pi i \theta}dF_\theta$.  $U$ has purely continuous spectrum and hence $L^{2}[0, 1] = \h_{0}(U)$. [Jacobs-Glicksburg-Deleeuw Theorem [3]]. $U$ is the Cayley transform of the self-adjoint operator $A = \int\limits^{+\infty}_{-\infty} \lambda ~dE_\lambda, $ and the cogenerator of the $C_{0}$ unitary group $\U : U_{t}= \int\limits^{+\infty}_{-\infty} e^{i t \lambda} ~dE_\lambda, -\infty < t < + \infty$ on $L^{2}[0, 1]$. Since $\h_{0}$ is reducing, $ \h_{0}(\C) = \h_{0}(\U)$, and from [9, 10, 6.2], $\h_{m}(\C)= \h_{0}(\C)$. Similarly $\h_{m}(\U)= \h_{0}(\U)$. Hence $\U ~\ma~\C$ are \emph{weakly Poisson recurrent} and $L^{2}[0, 1] = \h_{m}(\U)=\h_{m}(\C)$.

Let $W\{z_{0}, z_{1},... \} = \{0, z_{0}, z_{1},... \}$ be the unilateral shift on ${l^{2}(\mathbb{Z}_{+})}$. $ W$ is an isometry and \emph{completely non-unitary(cnu)} [7] and hence \emph{weakly stable}. Since $1$ is not an eigenvalue of $W$ and $W$ is isometric,  $W$ is the cogenerator of a $C_{0}$ isometric semigroup $\W = \{W_{t}\}$. Since $W$ is \emph{cnu} $\W$ cannot have a closed invariant subspace on which each $W_{t}$ is unitary, hence it is also \emph{cnu} and therefore \emph{weakly-stable}. [7, 9]  [16, IX.9],[2, 3.2 ]

Define a $C_{0}$ contraction semigroup on $\h$ by $\T: T_{t} = U_{t} \oplus W_{t}~\mf t\geq 0$. The cogenerator of $\T$ is $T = U \oplus W$, $\D =\{T^{n}: n \geq 0 \}$ its semigroup. Then $\h_p(\mathcal{T}) = \h_p(\mathcal{D})=\{0\}$, i.e. $\h = \h_0(\mathcal{T})=\h_0(\mathcal{D})$ consists of flight vectors[9]. The equivalence of weak stability will be proved in [Example 5.6].

\section{Dynamical Systems on $\h_{0}$ }
The limit states of $\T ~\ma ~ \D$ are determined by their action on $\h_{u}$ - the unitary space - so we consider the groups  $\U ~\ma ~ \C$.
In the following $\s$ will denote one of the above: $\mathcal{T} ,~\mathcal{D},~ \U, ~\mo~ \C$ which will be apparent from the context. To be self contained we collect some basic notions from [9,10]. Definitions and results will be stated for $\s =\T$ or $ \U$, the extension to $\s = \D \mo \C$ will be clear. We examine the dynamics of $\s$ with a classical eye.
\subsection{Dynamical Systems}
The semigroup $\s$ defines a dynamical system on $\h$ :
\begin{equation}\label{eqn1}
  x(t)=  T_t x_{0},~ x(-t)= T^{*}_t x_{0} ~\mf~ t \geq 0, ~\mfa ~ x_{0} \in \h.
\end{equation}
For (2.1) $\s$ has $\omega ~\ma~\alpha$ - \emph{limit operators}, $\Omega = \{V : V = \o- \lim  T_{t_{k}} ,~  t_{k}\uparrow + \infty \}$ and $\mathcal{A} = \Omega^{*}$. These operators define the $\omega$ and $\alpha$ -  \emph{limit states} of $x_{0}$, $ \Omega(x_{0})=\{y :y=\o-\lim  T_{t_{k}} x_{0},~ t_{k} \uparrow + \infty \} = \Omega \cdot x_{0}$ and  $ \mathcal{A}(x_{0})= \Omega^{*} \cdot x_{0}$. [8, 9, 10]. These sets are $\s, \s^{*}$ invariant.[9, 10]. The system (2.1) also has limit cycles as in the finite dimensional case of the Poincare- Bendixson theorem,
\begin{equation}\label{eqn2}
  y(t,x_{0}) = T_{t} V x_{0},~ y(-t,x_{0})= T^{*}_{t}V^{*} x_{0} ~\mf~ t \geq 0~\ma~ V \in \O.
\end{equation} $$$$
The trajectories of (2.2) converge weakly and pointwise to the limit cycles, i.e. given $x_{0}$ and $V = \o- \lim  T_{t_{k}}$ then
\begin{equation}\label{eqn3}
  \o -\lni x(s+t_{n},x_{0})= y(s,x_{0}).
\end{equation}

\subsection{Subspaces and Recurrence}
For $x \in \h_{0}$, $M(x,\s)=\overline{sp~\Omega}(x,S)$ is the \textit{limit subspace} for $x$. Note that $\O \cdot x =\Omega(x,\s) =\cap \{ \overline{\s(T_{s}x): s\geq 0}\}$ and ${T_{t}}^{*} \Omega(x)=\Omega(x)$ for all $t \geq 0$.[9, 10]

$\h_w(\s)$ is the collection of \textit{weakly stable states} - $\o-\lim\limits_{s\rightarrow \infty} T_s x = 0$. They form a a closed reducing subspace $\h_w(\s) = \bigcap \{ ker V : V \in \Omega_{\s}\}=\bigcap \{ ker V^{*} : V \in \Omega_{\s}\}=\h_w(\s^{*})$ [9, 10].

 A vector x in $\h_{0}$ is \emph{Poisson recurrent}(\emph{Pr}) if and only if  $x \in \Omega(x, \s)= \O \cdot x$, i.e. $x = Vx$ for some $V \in \Omega$. If $x$ is \emph{Pr} then $x \in \h_{u}$ and $\o -lim~ T_{t_{n}}x = x ~\m{implies}~ s-lim ~T_{t_{n}}x = x$.

A vector x in $\h_{0}$ is \emph{ weakly Poisson recurrent (wPr)} $\ii ~ x \in M(x, \s)$. $\h_m$ is the collection of \emph{wPr} vectors. Note that $x \in \h_{0}$ is (\emph{wPr}) if there is a net for this $x$ in $\overline{sp}~ \s$, $\{(\sum\limits^{n}_{k=1}a_{k}T_{s_k} )_{\alpha}  ~: \alpha \in \Delta \}$ with $\o-lim_{\Delta} (\sum\limits^{n}_{k=1}a_{k}T_{s_k} )_{\alpha}x = x$. Since $\h_{m}= \h_{0}\ominus \ \h_{w}$, $\h_m$ is a closed and reducing subspace and $\h_{m}(\s) =  \h_{m}(\s^{*})$. If $x$ is \emph{wPr} then $M(x, \s)$ is reducing and $M(x,\s)=\overline{sp~\Omega \cdot x}=\overline{sp~\Omega^{*} \cdot x} = M(x,\s^{*})$ [9, 3.4 -5], i.e. the future of a \emph{wPr} vector coincides with its past.

An ortho-normal set $\{x_{\tau}: \tau \in \Pi\}\subset \h_{m}(\s) $ is a \emph{recurrent spanning set} [9, 10] for $\s$ if $\h_{m}(\s) = \bigoplus_{\Pi} M_{\tau} ~ \mf ~ M_{\tau}= M(x_{\tau}, \s)$. From [9, Theorem 2.5] if $\h_{m}(\s) \neq 0$ then $\s$ has an ortho-normal  \emph{recurrent spanning set} $\{x_{\tau}: \tau \in \Pi\}\subset \h_{m}(\s) $.

\section{Limit Algebras}
Prompted by (2.1 - 2.2) the composition of the limit-cycles and the structure of the space $\h_{m}$ we consider the limit operators $\O_{\s}$ of $\s=\U $ or $ \C$ and the algebra they generate in $\L(\h_{u})$.
\subsection{Algebras}
For any commutative collection of operators $\mathcal{A} \subset L(\h_{u})$, let $\M^{\A}$ be the least *-closed, weakly closed sub-algebra of $L(\h_{u})$ containing $\mathcal{A}$. $\mathcal{A}$ is the generating set of $\M^{\A}$ and $P_{\A}$ is its unit.[14, 1.7],

In particular, for $\mathcal{A} = \O_{\s} ,~ \M^{\Omega_{\s}}$ is the $\emph{limit algebra}$ generated by the limit operators $\Omega_{\s}$ of $\s$. Note that for all $x \in \h, V \in \Omega_{\s}, Vx = Vx_{m}\in \h_{m} $ and hence $\M^{\Omega_{\s}}$ is a subalgebra of $L(\h_{m})$.

\subsection{The Generator Sets $\A$}

  On the unitary space $\h_{u}$ of $\U \ma \C$ let $\E = \{E_\lambda:-\infty<\lambda<+\infty\}$ and $\F = \{F_\theta:0\leq\theta< 1\}$ be their respective spectral families with $A$ the self-adjoint operator generating the unitary group $\U$. Note that $\mathcal{E} \ma \mathcal{F}$ satisfy $E_\lambda = F_{-2arcot(\lambda)}, -\infty<\lambda<+\infty$ [13, $\S$ 121] and hence $ \mathcal{E} = \mathcal{F} $.

   The spaces of Borel functions essentially bounded $\E,~ \F$ a.e., $\L^{\infty}_{\E}, \L^{\infty}_{\F}$ have corresponding algebras[13, $\S$ 109,129]
   $$\mathcal{L}_{\mathcal{E}} = \{u(A)=\int\limits^{+\infty}_{-\infty}u(\lambda)dE_\lambda: u \in \L^{\infty}_{\E}\},$$ and
   $$ \mathcal{L}_{\mathcal{F}}= \{w(U) =\int\limits^{1}_{0}w(e^{2 \pi i \theta})dF_\theta:w \in \L^{\infty}_{\F}\}. $$

  Since $ \mathcal{E} = \mathcal{F} $ we have $\overline{\mathcal{L}}_{\mathcal{E}}= \overline{sp}~\mathcal{E}=\overline{sp}~\mathcal{F} =  \overline{\mathcal{L}}_{\mathcal{F}}$ in $\L(\h_{u})$, and if $\h_{u}$ separable [13, $\S$ 129] $$\M^\mathcal{E} = \overline{\mathcal{L}}_{\mathcal{E}} = {\mathcal{L}}_{\mathcal{E}} = {\mathcal{L}}_{\mathcal{F}} = \overline{{\mathcal{L}}}_{\mathcal{F}}= \M^\mathcal{F}$$ with unit $P_{u}$ the orthogonal projection on $\h_{u}$.

The algebras  defining the dynamics of $\s$ have the following generating sets :
$ \mathcal{E}=\mathcal{F} $, the groups $\mathcal{U}$ and $\mathcal{C}$, and the sets of limit operators and closures - $\Omega_{\C}$, $\Omega_{\U}$, $\overline{\Omega_{\U}}$, $\overline{\Omega_{\C}}$.

$P_{u}$ is the unit of $\M^{\C}~\ma~\M^{\U}$ while the identities for $\h_{m}(\C) \ma \h_{m}(\U)$ are the units $P_{m_{c}} \in \M^{\O_{\C}}$ and $P_{m_{u}} \in \M^{\O_{\U}}$ respectively.

\subsection{Limit Algebras and Limit Spaces}
For all $x \in \h$, $M(x,\s) = M(x_{m},\s)$ and from (2.2) $M(x,\s)$ is $\s^{*}$ invariant, $\h_{m}(\s)=\h_{m}(\s^{*})$, and $P_{m}(\s)=P_{m}(\s^{*})$ [9, 3.4 - 5].

\begin{theorem} For the groups $\s = \U ~ \ma ~ \C$, and $ x \in \h$, then $$a)~M(x,\s)=\overline{sp}~\Omega_{\s} \cdot x=\M^{\O_{\s}}\cdot x = \M ^{\overline{\O_{\s}}}\cdot x = \overline{sp}~\overline{\O}_{\s} \cdot x = \mathcal{M}^{\s} P_{m_{\s}} \cdot x,$$ and $$b)~\M^{\overline{\O_{\s}}} = \M^{\O_{\s}} =\mathcal{M}^{\s} P_{m_{\s}}.$$ The unit of $\M^{\O_{\s}}$ is $P_{m_{\s}}$ the orthogonal projection onto $\h_{m}(\s)$.
\end{theorem}

\begin{proof}
  Statement a) follows from the above remarks, [Section 2], and [9,10]. Note that if $P_{m}$ is the projection on $\h_{m}$, $T_{t}P_{m}x~\ma~ T_{t}^{*}P_{m}x \in M(x,\s)$ [9, 3.4-5] and therefore for all $x \in \h_{0}, ~\M^{\O_{\s}}\cdot x \subset \mathcal{M}^{\s}\cdot P_{m}x \subset  M(x,\s)$ and hence $\M^{\Omega_{\s}} \cdot x =  \M^{\s}P_{m} \cdot x = M(x,\s)$.

For b), fix $x \in \h_{0}$ and $T \in \M^{\s}$. For this $x$ and any $\epsilon > 0$, since $M(x, \s)= \overline{sp~\Omega_{\s}\cdot x}$, there exists $\sum a_{k}V_{k} \in sp ~\Omega_{\s} \subset \M^{\Omega_{\s}}$ such that $\|\sum a_{k}V_{k}x - TP_{m}x\|< \epsilon$. Since $x \in \h_{0}~ \ma~\epsilon > 0$ were arbitrary $TP_{m} \in (\M^{\Omega_{\s}})^{a} = \M^{\Omega_{\s}} $ (the $W^{*}$ algebra $\M^{\Omega_{\s}}$ is strong-operator closed [14, \S  1.15.1]).
\end{proof}
\begin{theorem}For the subsets of $\L(\h_{u})$ of [Section 3.2], $\mathcal{C},~ \mathcal{U},~ \E,~ \F$,
$$\M^\mathcal{U}=\overline{\mathcal{L}_{\mathcal{E}}}=\M^\mathcal{E} = \M^\mathcal{F}=\overline{\mathcal{L}_{\mathcal{F}}} =\M^\mathcal{C},$$
is a $W^*$ subalgebra of $\L(\h_{u})$ with unit $P_{u}$ the identity of $\L(\h_{u})$.
\end{theorem}

\begin{proof}\

\begin{enumerate}

\item [1)]Since $ \U \subset \overline{sp}~\E = \overline{\L_{\E}}= \M^{\E}$, and $ \C \subset \overline{sp}~\F = \overline{\L_{\F}}= \M^{\F}$ then $\M^{\U} \subset \overline{\L _{\E }} = \M^{\E} = \M^{\mathcal{F}} = \overline{\L_{\F}} \supset \M^{\C}$  since $ \mathcal{F} = \mathcal{E}$, i.e. $$\M^{\U} \subset \overline{\L _{\E }} = \M^{\E}\ma \M^{\C}\subset \overline{\L_{\F}} =  M^\mathcal{F}. $$
\item [2)] To show : $\M^\mathcal{F}\subset \M^\mathcal{C}$:\\
Let $A$ be the self-adjoint generator of $\U$. As in the von Neumann construction of the spectral theorem for $A$ from the Cayley transform $U$ [13, $\S$ 109 and $\S$ 126], the unique spectral family $\mathcal{F}=\{F_\phi:0<\phi<1\}$ for $U$ is the strong operator limit of polynomials in $U \ma U^{*}$.  Hence $ \mathcal{F} \subset \M^\mathcal{C} $ and therefore $\M^\mathcal{F} = \overline{\L_{\F}}\subset\M^\mathcal{C}$.

\item [3)] $\M^\mathcal{C} \subset \M^\mathcal{U}$:\\
On the reducing subspace $\h_{u}$, $I = P_{u}$. Therefore as in [16, XI.4]  $(I - iA)^{-1} =\int\limits^{+\infty}_{0}e^{-t}U_t dt \in M^\mathcal{U}$. The closed operator $A( I - iA) ^{-1}$ has domain $\h_{u}$ and hence is bounded by the Closed Graph Theorem. Moreover for $x \in \h_{u}$,
\begin{equation}\label{eqn5}
 A(I - iA)^{-1}x = s-\lim\limits_{t \to 0}\frac{1}{t}( U_t - I)( I - iA)^{-1} x \in M^\mathcal{U}.
\end{equation}
Hence $(I+iA)( I - iA)^{-1}= (iI - A)(iI + A)^{-1} = U \in  M^\mathcal{U}$. Since $M^\mathcal{U}$ is *-closed, $U^*$ is also in $M^\mathcal{U}$ and hence $M^\mathcal{C} \subset M^\mathcal{U}$.

\item [4)]Combing the previous paragraphs\\
From 1) and 2): $\M^\mathcal{C} \subset \overline{\L_{\F}} = \M^\mathcal{F}\subset \M^\mathcal{C}$ and hence $\M^\mathcal{C} = \overline{\L_{\F}} = \M^\mathcal{F}$.\\
From 1) and 3): $\M^\mathcal{C} \subset \overline{\L_{\F}} = \M^\mathcal{F} \subset \M^\mathcal{C}$ implies
$$\M^\mathcal{C}=\M^\mathcal{F}=\overline{\L_{\F}} = \overline{\L_{\E}}= \M^\mathcal{E}\supset \M^\mathcal{U} \supset \M^\mathcal{C}.$$ Theorem 3.2  follows.
\end{enumerate}
\end{proof}

\begin{theorem}
For all $x, ~ M(x, \C)=\M^{\Omega_{\C}} \cdot x = \M^{\C}P_{m_{\C}} \cdot x$ is a separable Hilbert space and hence
$$M(x, \C)= M(x_{m}, \C)=\M^{\Omega_{\C}} \cdot x_{m}= \overline{\mathcal{L}_{\mathcal{E}}}\cdot x_{m} = \mathcal{L}_{\mathcal{E}}\cdot x_{m} = \mathcal{L}_{\mathcal{F}}\cdot x_{m}.$$
A similar statement holds for $\U$.
\end{theorem}
\begin{proof}
The group $\C $ is separable and from [Theorem 3.2]
$M(x, \C)=M(x_{m}, \C)=\M^{\Omega_{\C}} \cdot x_{m}=\overline{\mathcal{L}_{\mathcal{E}}} \cdot x_{m}.$
Hence $M(x_{m}, \C)=\M^{\C} \cdot x_{m}$ is separable. From [13, $\S$ 106], $\mathcal{L}_{\mathcal{E}}P_{m_{\C}}=\overline{\mathcal{L}_{\mathcal{E}}}P_{m_{\C}}$ in $\L(\M_{\tau})$ and therefore
$$M(x, \C)=\M^{\Omega} \cdot x_{m}= \overline{\mathcal{L}_{\mathcal{E}}} \cdot x_{m} = \mathcal{L}_{\mathcal{E}}\cdot x_{m}= \mathcal{L}_{\mathcal{F}}\cdot x_{m}.$$
\end{proof}

\section{Structure of $\h_{m}$}
  By [2.2 and Theorem 3.1 ] each $M(x_{\tau}, \C) = \mathcal{M}^{\Omega_{\C}} \cdot x_{\tau}$ is $U, U^{*}$ invariant. Hence if $M(x_{\tau}, \C) \neq \h$ the closed reducing subspace $M(x_{\tau}, \C)^{\perp}$ has a \emph{weakly wandering} \footnote{ There is a sequence $ 0 < k_{0}<k_{1}<... \in \mathbb{Z}^{+}$ for which $<U^{k_{j}} x, U^{k_{l}} x >=0 \mf k_{j}\neq k_{l}$.} vector for $U$ by Krengel's Theorem [5]. Using the construction of [9, 10] an orthonormal set $\{x_{\tau}\}$ can be chosen \emph{weakly wandering} for $U$ such that
  \begin{equation}\label{eqn6}
   \h_{m}(\C) = \sum_\tau M(x_{\tau}, \C) = \sum_\tau \M^{\Omega_{\C}} \cdot x_{\tau}.
 \end{equation}

 Hence the limit cycles for $\D$ and its dynamical system
 \begin{equation}\label{eqn7}
   x(n,x_{0}) = T^{n}x_{0}, ~x(-n,x_{0}) = {T^{*}}^{n}x_{0}~ n\geq 0
 \end{equation}

are defined by $U$ and $\M^{\Omega_{\C}}$ :
\begin{equation}\label{eqn8}
  y(n,x_{0}) = T^{n} V x_{0} =  T^{n} V x_{m} = \sum T^{n} V T_{\tau} x_{\tau}.
\end{equation}
  for $T_{\tau} \in \M^{\Omega_{\C}}$ and for all $ V \in \Omega_{\C}.$ For a semigroup $\s = \T \mo \D$ (4.1) characterizes the flight vectors [1, 2, 9, 10].
  \begin{equation}\label{eqn8}
    x_{0} = \sum T_{\tau} x_{\tau} + x_{w}, T_{\tau} \in \M^{\Omega_{\C}}, x_{w} \in \h_{w} .
  \end{equation}

\section{Entangled Systems and Weak Stability}
The results of [4.0] lead us to ask when do  the limit cycles of $\U$ and the cogenerator group $\C$  approximate each other? We formalize this question:

\begin{definition}$\U~\ma~\C$ are \emph{ entangled } on $\h$ if for all $x \in \h$, $\Omega_{\U}\cdot x \subset M(x, \C)$ and conversely $\Omega_{\C}\cdot x \subset M(x, \U)$. They are \emph{decoupled} if $\h_{m}(\mathcal{S})\cap \h_{m}(\mathcal{U}) = \{0\}$.
\end{definition}
\begin{remark}
Some observations:
\begin{itemize}

   \item 1) Assume $\mathcal{U} \ma \mathcal{C}$ are \emph{entangled}. For all $x \in \h, ~  M(x, \U) = M(x, \C)$ since from [3.1]:
             $$M(x, \C)= \overline{sp~\Omega_{\C}\cdot x} \subset M(x,\U)~\ma~ M(x,\U)= \overline{sp~\Omega_{\U}\cdot x} \subset M(x, \C).$$

    \item 2) It follows from 1) that if $\U ~ \ma ~\C$ are \emph{entangled} they have a common \emph{wPr} subspace $\h_{m} = \h_{m}(\U)=\h_{m}(\C)$ with orthogonal projection $P_{m}$.

    \item 3) Lemma: If $\U \ma \C$ are \emph{ entangled } they have a common \emph{limit-algebra} $\M^{\Omega}$.\\
    Proof: Fix $\varepsilon>0 ~\ma~x \in \h_{m}$ - the common \emph{wPr} subspace of 2). Let $T \in \M^{\Omega_{\U}}$ and use the argument of [ Theorem 3.1]. $\M^{\Omega_{\U}} \cdot x = M(x, \U) = M(x, \C)= \M^{\Omega_{\C}} \cdot x $. For this $x$ and any $\epsilon > 0$, since $M(x, \C)= \overline{sp~\Omega_{\C}\cdot x}$, there exists $\sum a_{k}V_{k} \in sp ~\Omega_{\C} \subset \M^{\Omega_{\C}}$ such that $\|\sum a_{k}V_{k}x - Tx\|< \epsilon$.
    Since $x \in \h_{m}~ \ma~\epsilon > 0$ were arbitrary $T \in (\M^{\Omega_{\C}})^{a} = \M^{\Omega_{\C}} $ (the $W^{*}$ algebra $\M^{\Omega_{\C}}$ is strong-operator closed) and $\M^{\Omega_{\U}} \subset \M^{\Omega_{\C}} $.
       Interchanging $\U \ma \C$ in the above argument yields the common \emph{limit-algebra} $\M^{\Omega} = \M^{\Omega_{\U}}=\M^{\Omega_{\C}}$ with unit $  P_{m}$.

   \item 4) Conversely, suppose $\U \ma \C$ have a common \emph{limit-algebra} $\M^{\Omega}=\M^{\Omega_{\C}}=\M^{\Omega_{\U}}$  with unit $P =  P_{m_{\U}} = P_{m_{\C}}$. Then $\h_{m}(\C)=P_{m_{\C}}\h = P_{m_{\U}}\h = \h_{m}(\U)$. By [Theorem 3.1] for all $x\in \h$,
   $$M(x, \C)=\M^{\Omega_{\C}}x=\M^{\Omega_{\U}}x= M(x, \U)$$
   and therefore $\U \ma \C$ are \emph{entangled}.
\end{itemize}
\end{remark}
Conclusion:

\begin{theorem} For the groups $\U \ma \C$ T.F.A.E
\begin{itemize}
  \item a) $\U \ma \C$ are \emph{entangled},
  \item b) $\U \ma \C$ have a common \emph{limit-algebra} $\M^{\Omega}=\M^{\Omega_{\C}}=\M^{\Omega_{\U}}$ with unit $P_{m}$,
  \item c)$\h_{m}(\mathcal{C})=\h_{m}(\mathcal{U})$.
\end{itemize}
\end {theorem}
\begin{proof} We need only show c) implies a). If $\h_{m}(\mathcal{C})=\h_{m}(\mathcal{U})$ then they have a common orthogonal projection $P_{m_{\C}} = P_{m_{\U}}$. Hence from [Theorem 3.1, 2] for all $x \in \h$, $$M(x, \C)= M(x_{m}, \C)= \M^{\O_{\C}}\cdot x_{m} = \M ^{\overline{\O_{\C}}}P_{m_{\C}}\cdot x = \overline{\L_{\E}}P_{m_{\C}}\cdot x = \overline{\L_{\F}}P_{m_{\U}}\cdot x= \M ^{\overline{\O_{\U}}}P_{m_{\U}}\cdot x = M(x,\U),$$
and hence a) follows.
\end{proof}

\begin{cor}
Since $\U \ma \C$ are separable they are \emph{entangled} if and only if  $$ \M^{\O_{\U}} =\L_{\E}P_{m} = \L_{\F}P_{m} =\M^{\O_{\C}}.$$
\end{cor}

\begin{remark}
  By the $\h_{0}$ splitting theorem of [9 Theorem 2.5] $\h_{0}= \h_{m} \oplus \h_w$. Hence as a result of  [Theorem 5.3] : The weak-stability of a $C_{0}$ contraction semigroup $\T$ is equivalent to that of its cogenerator $T$ ($\h_w(\T) = \h_w(\D)$ ) if and only if their unitary parts $\U ~ \ma ~\C$ have a common limit algebra (are \emph{entangled}). This addresses Open Question 2.23 of T. Eisner [2, 2.23, p176].
\end{remark}

\begin{remark}
  When $\U ~ \ma ~\C$ are \emph{entangled} their limit cycles can be expressed in terms of each other. For example if $\U ~ \ma ~\C$ are \emph{entangled} and $\h_{m}(\C) = \sum_\tau M(x_{\tau}, \C) = \sum_\tau \M^{\Omega_{\C}} \cdot x_{\tau}$ is the expansion of [4.0], then for $x_{0} \in \h,~ V \in \O_{U},$ the limit cycles for (2.1) have the form $$y(t,x_{0}) = T_{t}V x_{0}=\sum T_{t} V T_{\tau} x_{\tau} $$
  for $T_{\tau} \in \M^{\O}$ the common limit algebra.
\end{remark}
\begin{example}
The Limit Algebras of [1.3]\\
a) Consider the semigroup of [1.3] $\T: T_{t} = U_{t} \oplus W_{t}~\mf t\geq 0$ with cogenerator $T = U \oplus W$ on  $\h = L^{2}[0, 1] \oplus  l^{2}(\mathbb{Z}^{+})$.

$ U$ is the unitary operator of [ 9, 10] with defining spectral family $ \F = \{F_\theta\}$  on $L^{2}([0,1])$ with group $\C$. It is the cogenerator of the group $\U =\{U_{t}\}$ of [9, 10]. Using the argument of [9, Theorem 6.2] there exist subsequences $\{2^{m_{k}}\} ~\ma ~\{2^{n_{j}}\}$ with limit operators $V \in \Omega_{\C} ~ \ma ~ W \in \Omega_{\U} $ for which $\o-\lim U^{2^{m_{k}}} = V ~\ma\o-\lim U_{ 2 \pi 2^{n_{j}}} = W~ \ma ~ W=V=I.$

b) The above implies $U = UI = UV \in \Omega_{\C}~\ma~ U_{t} = U_{t}I = U_{t}V \in \Omega_{\U} $. From the construction of [Theorem 3.5] $U_{t} \in \M^{\Omega_{\C}} ~ \ma ~ U \in \M^{\Omega_{\U}}$. These observations imply:
$$\M^{\Omega_{\C}} \subset \M^{\C} \subset \M^{\Omega_{\U}} ~\ma~ \M^{\Omega_{\U}} \subset \M^{\U} \subset \M^{\Omega_{\C}}$$
and hence there is a common limit algebra and by [Theorem. 5.3] $\h_{m}(\C) = \h_{m}(\U)$,i.e. $\U~\ma~\C$ are \emph{entangled}.

$$\mathcal{M}^{\Omega_\C}= \mathcal{M}^{\C}=\overline{\L_{\E}}= \overline{\L_{\F}}=\M^\mathcal{U}=\mathcal{M}^{\Omega_\U}\equiv\M^{\Omega}.$$

c) Since $\h = L^{2}([0,1])$ is separable: $$\mathcal{M}^{\Omega_\C}= \overline{\L_{\E}}= \L_{\E}= \L_{\F}= \overline{\L_{\F}}=\mathcal{M}^{\Omega_\U}\equiv\M^{\Omega}.$$
Moreover for each $f \in L^{2}([0,1]), ~ f = If = Vf =Wf \in \h_{m}(\C)\cap\h_{m}(\U)$,i.e. $\h_{m}(\C)= \h_{m}(\U) = L^{2}([0,1])$.

d) Suppose for $\overline{x} = (f, \overline{z}) \in \h ~ \o- \lim\limits_{t\to\infty} T_{t}\overline{x}=0$. Then each of its components converge weakly to 0. But since $\h_{w}(\U) = \h_{w}(\C) = 0$ from c), $f = 0$ and hence $\h_{w}(\T) = \{0\}\oplus l^{2}_{+} = \h_{w}(\mathcal(D))$ and $\h_{m}(\T) = \h_{m}(\mathcal(D))$, i.e. $\T$ and $\mathcal{D}$ of [1.3] are \emph{entangled} and hence \emph{weakly stable equivalent}. The common limit algebra of $\T: T_{t} = U_{t} \oplus W_{t}$ and cogenerator $T = U \oplus W$ in $\L(\h)$ is  $\M^{\Omega_{\T}}=\M^{\Omega_{\D}}= \L_{\E}\oplus \{0\}$.
\end{example}

\section{Bibliography}
\begin{itemize}

  \item [1] T. Eisner, Stability of Operators and Operator Semigroups, Preprint, Mathematisches Institut, Universitat Tubingen, Tuingen, Germany..
  \item [2] T. Eisner, Stability of Operators and Operator Semigroups, Operator Theory: Advances and Applications 209, 1st Edition. 2010. Buch. vIII, 204 S. Hardcover ISBN 978 3 0346 0194 8.
  \item[3] I. Glicksberg and K. Deleuw, Applications of Almost Periodic Compactifications, Acta. Math., 105 (1961) 63-97.
  \item[4] E. Hille, R. Phillips, Functional Analysis and Semigroups, American Mathematical Society, (1946).
  \item[5] U. Krengel, Weakly wandering vectors and weakly independent partitions, Amer. Math. Soc. Trans. 164 (1972), 199-226.
  \item[6] B. Nagy and C. Foias, Harmonic Analysis of Operators on Hilbert Space, North Holland Publishing Company, Amsterdam, 1970.
  \item[7] B. Nagy and C. Foias, Functional Analysis, Frederick Ungar Publishing Company, New York, 1971.
  \item[8] V.V. Nemitsky and V.V. Stepanov (1960) \emph{Qualitative Theory of Differential Equations}, Princeton University Press, Princeton, New Jersey.
  \item[9] R. E. O'Brien, Semigroup Dynamics for Flight Vectors, Inter. J. Dyn. Sys. Diff. Eq., to appear Fall 2020.
  \item[10] R. E. O'Brien, Almost Weakly Stable Contraction Semigroups are Weakly Poisson Recurrent, Preprint, •	DOI:10.13140/RG.2.1.4850.6725.
  \item[11] R. E. O'Brien, Flight Vectors and Limit Operators-an Example ,Research Gate Preprint, August 2018 DOI: 10.13140/RG.2.2.28268.62084.
  \item[12] R. E. O'Brien, The Limit Algebra And Weak Stability For C 0 Contraction Semigroups And Cogenerators ,Research Gate Preprint, November 2019 DOI: •	10.13140/RG.2.2.33374.79686.
  \item[13] F. Riesz, B. Sz-Nagy, Functional Analysis, Frederick Ungar Publishing Co. New York 1955.
  \item[14] S. Sakai, $C^{*}$ - Algebras and $W^{*}$ - Algebras, Springer Verlag, New York, 1971.
  \item[15] I. Segal, R. Kunze, Integrals and Operators, McGraw-Hill Book Company,  New York, 1968.
  \item[16] K. Yosida, Functional Analysis, Second Edition, Springer Verlag, New York, 1968.

\end{itemize}

\end{document}